\documentclass[11pt]{amsart}
\usepackage{array}
\usepackage{amsmath,amssymb,graphics,graphicx}

\begin{document}

\newtheorem{thm}{Theorem}[section]
\newtheorem{lem}[thm]{Lemma}
\newtheorem{prop}[thm]{Proposition}
\newtheorem{cor}[thm]{Corollary}
\newtheorem{defn}[thm]{Definition}
\newtheorem*{remark}{Remark}
\newtheorem{theorem}{Theorem}[section]

\newtheorem{lemma}[thm]{Lemma}
\newtheorem{corollary}[thm]{Corollary}
\newtheorem{observation}[thm]{Observation}
\newtheorem{conj}[thm]{Conjecture}
\newtheorem{definition}[thm]{Definition}

\numberwithin{equation}{section}

\newcommand{\Z}{{\mathbb Z}} 
\newcommand{\Q}{{\mathbb Q}}
\newcommand{\R}{{\mathbb R}}
\newcommand{\C}{{\mathbb C}}
\newcommand{\N}{{\mathbb N}}
\newcommand{\FF}{{\mathbb F}}
\newcommand{\fq}{\mathbb{F}_q}

\newcommand{\rmk}[1]{\footnote{{\bf Comment:} #1}}

\renewcommand{\mod}{\;\operatorname{mod}}
\newcommand{\ord}{\operatorname{ord}}
\newcommand{\TT}{\mathbb{T}}
\renewcommand{\i}{{\mathrm{i}}}
\renewcommand{\d}{{\mathrm{d}}}
\renewcommand{\^}{\widehat}
\newcommand{\HH}{\mathbb H}
\newcommand{\Vol}{\operatorname{vol}}
\newcommand{\supp}{\operatorname{supp}}
\newcommand{\area}{\operatorname{area}}
\newcommand{\tr}{\operatorname{tr}}
\newcommand{\norm}{\mathcal N} 
\newcommand{\intinf}{\int_{-\infty}^\infty}
\newcommand{\ave}[1]{\left\langle#1\right\rangle} 
\newcommand{\Var}{\operatorname{Var}}
\newcommand{\Prob}{\operatorname{Prob}}
\newcommand{\sym}{\operatorname{Sym}}
\newcommand{\disc}{\operatorname{disc}}

\newcommand{\cond}{\operatorname{cond}} 
\newcommand{\lcm}{\operatorname{lcm}}
\newcommand{\Kl}{\operatorname{Kl}} 
\newcommand{\leg}[2]{\left( \frac{#1}{#2} \right)}  

\newcommand{\sumstar}{\sideset \and^{*} \to \sum}

\newcommand{\LL}{\mathcal L} 
\newcommand{\sumf}{\sum^\flat}
\newcommand{\Hgev}{\mathcal H_{2g+2,q}}
\newcommand{\USp}{\operatorname{USp}}
\newcommand{\conv}{*}
\newcommand{\dist} {\operatorname{dist}}
\newcommand{\CF}{c_0} 
\newcommand{\kerp}{\mathcal K}

\renewcommand{\baselinestretch}{1.2}
\newcommand{\ignore}[1]{}
\newcommand{\monic}{{\operatorname{monic}}}
\newcommand{\pr}{{\operatorname{prime}}}
\newcommand{\dis}{{\operatorname{distinct}}}
\newcommand{\mob}{{\operatorname{o}}}
\def\square{\vrule height6pt width7pt depth1pt}
\def\endpf{\hfill\square\bigskip}


\newcommand{\beq}{\begin{equation}}
\newcommand{\eeq}{\end{equation}}
\newcommand{\lb}{\left(}
\newcommand{\rb}{\right)}
\newcommand{\rf}[1]{(\ref{#1})}
\renewcommand{\ss}{\subset}

\newcommand{\sr}{{\mathcal{S}(\R)}}
\newcommand{\hgq}{\mathcal{H}(2g+1,q)}
\newcommand{\av}[1]{\langle #1\rangle}
\newcommand{\pin}{\Pi_n}
\newcommand{\uf}{{\underline{F}}}
\newcommand{\ug}{{\underline{G}}}
\newcommand{\uU}{{\underline{U}}}
\newcommand{\uv}{{\underline{V}}}
\newcommand{\uw}{{\underline{W}}}
\newcommand{\uo}{\underline{O}}
\newcommand{\ft}{\widetilde{f}}
\newcommand{\fh}{\widehat{f}}
\newcommand{\ut}{\widetilde{U}}
\newcommand{\ity}{\infty}
\renewcommand{\d}{\mathrm{d}}
\newcommand{\uh}{\hat{U}}
\newcommand{\uu}{{\hat{u}}}
\newcommand{\mm}{{\mathbf{m}}}
\newcommand{\n}{{\mathbf{n}}}

\renewcommand{\th}{\theta}
\newcommand{\Lam}{\Lambda}
\newcommand{\del}{\delta}
\newcommand{\ka}{\kappa}
\newcommand{\al}{\alpha}
\newcommand{\sig}{\sigma}
\newcommand{\eps}{\epsilon}

\newcommand{\AES}{\mathcal T}

\title[quadratic Dirichlet L-functions, hyper-elliptic curves and RMT]
{Low-lying zeros of quadratic Dirichlet L-functions, hyper-elliptic
curves and Random Matrix Theory  }

\author{Alexei Entin, Edva Roditty-Gershon and Ze\'ev Rudnick}

\address{Raymond and Beverly Sackler School of Mathematical Sciences,
Tel Aviv University, Tel Aviv 69978, Israel}

\thanks{The research leading to these results has received funding from the European
Research Council under the European Union's Seventh Framework
Programme (FP7/2007-2013) / ERC grant agreement n$^{\text{o}}$
320755.}
\date{\today}

\begin{abstract}
The statistics of low-lying zeros of quadratic Dirichlet L-functions
were conjectured by Katz and Sarnak to be given by the scaling limit
of eigenvalues from the unitary symplectic ensemble. The $n$-level
densities were found to be in agreement with this in a certain
neighborhood of the origin in the Fourier domain by Rubinstein in
his Ph.D. thesis in 1998. An attempt to extend the neighborhood was
made in the Ph.D. thesis of Peng Gao (2005), who under GRH gave the
density as a complicated combinatorial factor, but it remained open
whether it coincides with the Random Matrix Theory factor. For
$n\leq 7$ this was recently confirmed by Levinson and Miller. We
resolve this problem for all $n$, not by directly doing the
combinatorics, but by passing to a function field analogue, of
L-functions associated to hyper-elliptic curves of given genus $g$
over a field of $q$ elements. We show that the answer in this case
coincides with Gao's combinatorial factor up to a controlled error.
We then take the limit of large finite field size $q\to \infty$ and
use the Katz-Sarnak equidistribution theorem, which identifies the
monodromy  of the Frobenius conjugacy classes for the hyperelliptic
ensemble with the group $\USp(2g)$. Further taking the limit of
large genus $g\to \infty$ allows us to identify Gao's combinatorial
factor with the RMT answer.
\end{abstract}
 \maketitle

\section{introduction}

\subsection{One-level densities for quadratic L-functions}
Our goal in this paper is to study statistics of low-lying zeros of
quadratic Dirichlet L-functions. To simplify the discussion, we
restrict to discriminants of the form $8d$, where $d>0$ is an odd,
square-free integer. The corresponding quadratic characters
$\chi_{8d}$ are then all primitive and even, and have conductor
$8d$. Denote the nontrivial zeros of  the corresponding L-function
$L(s,\chi_{8d})$ by
\begin{equation}
\frac 12 +\i\gamma_{8d,j}, \quad j=\pm 1,\pm 2, \dots
\end{equation}
where the labeling is so that $\gamma_{8d,-j} = -\gamma_{8d,j}$.
The number $N(T,8d)$ of such zeros with $0\leq \Re\gamma_{8d,j} \leq
T$ is asymptotically, for $T>1$,
\begin{equation}\label{N(T)}
N(T,8d) = \frac {T}{2\pi} \log \frac{8d T}{2\pi} -\frac T{2\pi} +O(
\log 8dT ).
\end{equation}

 We wish to study statistics of the zeros of $L(s,\chi_{8d})$ for random $d$. To do so, set
\begin{equation}
\mathcal D(X) = \{ X\leq d \leq 2X: d \mbox{ odd, square-free} \}
\end{equation}
Then $\#\mathcal D(X) \sim \frac 4{\pi^2} X$, as $ X\to \infty$. To
define what it means to pick a "random" discriminant from $\mathcal
D(X)$, we take a smooth  weight function $\Phi\geq 0$ supported in
the interval $(1,2)$, satisfying $\int \Phi(u)du=1$, and
define an averaging operator 
for  functions $f$ on $\mathcal D(X)$ by
\begin{equation}\label{smooth ensemble}
\ave{f}_{\mathcal D(X)}:=\frac 1{\#D(X)}   \sum_{d\in \mathcal D(X)}
\Phi(\frac dX) f(d).
\end{equation}
Thus we obtain a probability measure on   $\mathcal D(X)$ which
endows it with the structure of a a probability space (ensemble),
which we call the quadratic ensemble.

To count the number of zeros on the scale of the mean spacing $\log
8d/2\pi$ between the low-lying zeros, we define the linear
statistic, or one-level density, by taking an even Schwartz function
$f(r)$, which is analytic in a strip $|\Im r|\leq 1/2$, and setting
for $d\in \mathcal D(X)$
\begin{equation}
W_f(d):=\sum_{j } f(L\gamma_{8d,j}).
\end{equation}
Here $L=\log X/2\pi$.

The expectation values of the one-level densities for the quadratic
ensemble were studied by Katz and Sarnak \cite{KS-BAMS,
Sarnakappendix} (see also \cite{OS1, OS2}) who showed that, assuming
GRH, in the "scaling limit" $X\to \infty$, their expected value
coincides with the analogous quantity for the eigenphases of random
matrices from unitary symplectic groups $\USp(2g)$ in the limit
$g\to \infty$, that is
\begin{equation}\label{one-level conj}
\lim_{X\to \infty} \ave{W_f}_{\mathcal D(X)} =\intinf
f(x)\left(1-\frac{\sin 2\pi x}{2\pi x} \right)dx
\end{equation}
under the condition  that the Fourier transform $\^f(u) =
\int_{\R}f(x)e^{-2\pi \i xu}dx$ is supported in the interval
\begin{equation}\label{KS-OS}
|u|<2.
\end{equation}
The Density Conjecture \cite{KS-BAMS} is that \eqref{one-level conj}
holds for any test function $f$. See \cite{Stopple} for numerical
support for the conjecture and \cite{Miller} for a refined version.

\subsection{Higher moments and the $n$-level densities}

We want to study the moments of the linear statistic. The goal is to
show that
in the scaling limit the moments coincide with the analogous
quantity for the eigenphases of random matrices from unitary
symplectic groups.


The moments are determined by 
multi-linear statistics known as the $n$-level densities. To define
these, one starts with a Schwartz function $f\in \mathcal S(\R^n)$,
which is
even  in all variables.
The $n$-level density for $d\in \mathcal D(X)$ is
\begin{equation}
W_f^{(n)}(d) := \sum_{\substack{ j_1,\dots, j_n =\pm 1,\pm 2,\dots \\
|j_k|\mbox{ distinct}}} f(L \gamma_{8d,j_1},\dots, L\gamma_{8d,j_n})
\;,
\end{equation}
where  the sum is over $n$-tuples of indices $j_1,\dots ,j_n=\pm
1,\pm 2,\dots $ with $j_r\neq \pm j_s$ for $r\neq s$, and $L=\log
X/2\pi$. The density conjecture \cite{KS} for low lying zeros of
this family of L-functions is that the scaling limit coincides with
the scaling limit of the $n$-level densities for random matrices in
the unitary symplectic group $\USp(2g)$, that is
\begin{equation}\label{n-level conj}
 \lim_{X\to\infty} \ave{W_f^{(n)}}_{\mathcal D(X)} 
 =\int_{\R^n} f(x) W^{(n)}_{\USp}  (x) dx,
\end{equation}
where
\begin{equation}\label{RMT answer}
\begin{split}
W^{(n)}_{\USp}(x) &= \det (K(x_i,x_j))_{i,j=1,\dots,n},\\
 K(x,y) &=  \frac{\sin \pi(x-y)}{\pi(x-y)}- \frac{\sin \pi(x+y)}{\pi(x+y)}.
\end{split}
\end{equation}


The higher densities for this ensemble were investigated in the
Ph.D. thesis of Mike Rubinstein \cite{Rubinsteinthesis, Rubinstein},
who assuming GRH established \eqref{n-level conj} under the
condition that the Fourier transform
 $\^f(u) = \int_{\R^n}f(x)e^{-2\pi \i x \cdot u}dx$  is supported in the set
\begin{equation}\label{Gaussian range}
\sum_{j=1}^n |u_j| <1 \;.
\end{equation}
Note that for $n=1$, \eqref{Gaussian range} is only half the range
in \eqref{KS-OS}.

In his Ph.D. thesis \cite{Gao, Gaoarxiv}, Peng Gao attempted to
double the range in Rubinstein's result. He showed, assuming GRH,
that if $f$ is of the form $f(x_1,...,x_n)=\Pi_{j=1}^nf_j(x_i)$ and
each $\^f_j$ is supported in the range $|u_j|<s_j$ with $\sum s_j<2$
so that $f$ is supported on the range
\begin{equation}\label{KS range}
\sum_{j=1}^n |u_j| <2,
\end{equation}
then
\begin{equation}\label{Gao's result}
\ave{W_f^{(n)}}_{\mathcal D(X)} = A(f) + o(1),\quad X\to \infty,
\end{equation}
where $A(f)=A(f_1,...,f_n)$ is a complicated combinatorial
expression, taking almost a page to write down (see
Theorem~\ref{finalthm}). In view of \eqref{Gao's result}, proving
\eqref{n-level conj} in this range is reduced to a purely
combinatorial problem, of proving an identity
\begin{equation}\label{goal identity}
A(f) = \int_{\R^n} f(x) W^{(n)}_{\USp}  (x) dx
\end{equation}
which Gao verified for $n=2,3$.
More recently, Levinson and Miller \cite{LM} have confirmed
\eqref{goal identity} for $n=4,5,6,7$, aided by a machine
calculation. In this paper we confirm the equality for all $n$.
\begin{thm}\label{main thm quadratic} Assume GRH.  For
test functions whose Fourier transform $\^f$ is supported in the
region $\sum_{j=1}^n |u_j| <2$, we have
\begin{equation}
 \lim_{X\to\infty} \ave{W_f^{(n)}}_{\mathcal D(X)} = 
 \int_{\R^n} f(x) W^{(n)}_{\USp}  (x) dx.
\end{equation}
\end{thm}

Instead of directly attacking the combinatorial problem, we approach
it by comparing the densities of the zeros  with a function field
analogue, of zeros of L-functions for hyperelliptic curves of genus
$g$ defined over a finite field $\fq$. We then use the
equidistribution results of Deligne and Katz-Sarnak to pass to the
large finite field limit $q\to \infty$ and identify the limit with
RMT. This is similar in spirit to one of the ingredients in the work
of Ng$\hat{\mbox{o}}$  on the ``Fundamental Lemma'', where a
complicated combinatorial identity arising from a number field setup
is proved via a passage to the function field setting \cite{Ngo}. To
explain how we do it, we first describe the RMT context and then
move on to the function field setting.

\subsection{Random Matrix Theory (RMT)}
 For any continuous function $F$ on the set of
conjugacy classes of $\USp(2g)$, we denote by $\ave{F}_{\USp(2g)}$
its average with respect to the Haar probability measure on
$\USp(2g)$:
\begin{equation}
\ave{F}_{\USp(2g)} = \int_{\USp(2g)} F(U)dU.
\end{equation}

Recall that for  a unitary symplectic matrix $U\in \USp(2g)$, if
$e^{\i \theta}$ is an eigenvalue then so is $e^{-\i \theta}$. We can
then label the eigenvalues of $U$ as $e^{\i \theta_{\pm j}}$,
$j=1,\dots,g$ with the eigenphases $\theta_1,\dots,
\theta_g\in[0,\pi]$ and $\theta_{-j} = -\theta_j$.


To define $n$-level densities, one starts with a Schwartz function
$f\in \mathcal S(\R^n)$, which is
even in all variables,
and sets
\begin{equation}
\ft(\theta) = \sum_{m\in \Z^n} f\left(\frac g\pi(\theta +2\pi
m)\right) \;,
\end{equation}
which is $2\pi$-periodic and localized on a scale of $1/g$. The
$n$-level density is
\begin{equation}
W_f^{(n)}(U) = \sum_{\substack{ j_1,\dots, j_n =\pm1,\dots,\pm g \\
|j_k|\mbox{ distinct}}} \ft(\theta_{j_1},\dots, \theta_{j_n}) \;,
\end{equation}
where  the sum is over $n$-tuples of indices $j_1,\dots ,j_n = \pm
1,\dots,\pm g$ with $j_r\neq \pm j_s$ if $r\neq s$.

If we restrict the Fourier transform $\^f(u)$ to be supported in the
region
$|u|<\frac 1n$
then the first $n$ moments of the linear statistic $W^{(1)}_f$ in
RMT are Gaussian \cite{HR}.
This was called ``mock-Gaussian'' behavior in \cite{HR}. The higher
moments are also known, but no longer have a simple expression
(however see \cite{Hughes-Miller} for some nice expressions for the
centered moments of orthogonal families). It is the $n$-level
density which has a clean expression: In the scaling limit, the
$n$-level densities
are given by
\begin{equation}\label{n-level RMT}
   \lim_{g\to \infty} \ave{W_f^{(n)}}_{\USp(2g)} =
 \int_{\R^n} f(x) W^{(n)}_{\USp}  (x) dx,
\end{equation}
where $W^{(n)}_{\USp}$ is given by \eqref{RMT answer}.

\subsection{The hyperelliptic ensemble}

For a finite field $\fq$ of odd cardinality $q$ consider the family
$\mathcal H(2g+1,q)$ of all curves given in affine form by an
equation
$$ C_h: y^2 = h(x)$$
where
$$ h(x) = x^{2g+1} +a_{2g}x^{2g}+\dots +a_0 \in \fq[x]$$
is a square-free, monic polynomial of degree $2g+1$. The curve $C_h$
is thus nonsingular and of genus $g$. We consider $\mathcal
H(2g+1,q)$ as a probability space (ensemble) with the uniform
probability measure, so that the expected value of any function $F$
on $\mathcal H(2g+1,q)$ is defined as
\begin{equation}
\ave{F}_{\mathcal H(2g+1,q)}:=\frac 1{\#\mathcal H(2g+1,q)}
\sum_{h\in \mathcal H(2g+1,q)} F(h).
\end{equation}

The zeta function associated with the hyperelliptic curve $C_h\in
\mathcal H(2g+1,q)$ has the form
\begin{equation}\label{zeta function}
 Z_h(u) = \frac{\det(I-u\sqrt{q}\Theta_h)}{(1-u)(1-qu)}
\end{equation}
for a unique conjugacy class of $2g\times 2g$ unitary symplectic
matrices $\Theta_h\in \USp(2g)$ so that the eigenvalues
$e^{i\theta_j}$ of $\Theta_h$ correspond to zeros
$q^{-1/2}e^{-i\theta_j}$ of $Z_h(u)$. The matrix (or rather the
conjugacy class) $\Theta_h$ is called the unitarized Frobenius class
of $C_h$.
Katz and Sarnak showed \cite{KS} that as $q\to \infty$, the
Frobenius classes $\Theta_h$ become equidistributed in the unitary
symplectic group $\USp(2g)$: For any continuous function on the
space of conjugacy classes of $\USp(2g)$,
\begin{equation}\label{KS equidistribution}
\lim_{q\to\infty}  \ave{F(\Theta_h)}_{\mathcal H(2g+1,q)} =
\ave{F(U)}_{\USp(2g)}.
\end{equation}
This implies that various statistics of the eigenvalues can, in this
limit, be computed by integrating the corresponding quantities over
$\USp(2g)$. In particular, the $n$-level densities for the
hyper-elliptic ensemble $\mathcal H(2g+1,q)$ when $g$ is fixed are
given in the large finite field limit by
\begin{equation}\label{n-level KS}
\lim_{q\to \infty} \ave{W_f^{(n)}}_{\mathcal H(2g+1,q)}
=\ave{W_f^{(n)}}_{\USp(2g)}.
\end{equation}
Therefore, on further taking the large genus  limit $g\to \infty$
one gets
\begin{equation}\label{iterated limit}
\lim_{g\to \infty} \Big(\lim_{q\to \infty}\ave{W_f^{(n)}}_{\mathcal
H_{2g+1}} \Big) 
=\int_{\R^n} f(x) W^{(n)}_{\USp}(x) dx.
\end{equation}

\subsection{Comparing the hyperelliptic and quadratic ensembles}
We will compute the averages of the $n$-level densities for the
hyper-elliptic ensemble. We will show that in the range \eqref{KS
range} they are asymptotically equal to a complicated combinatorial
expression up to a remainder term that is negligible  for large $g$,
the same expression $A(f)$ which appears in Gao's result
\eqref{Gao's result}.
\begin{thm}\label{comparison thm}
Assume that $f(x_1,...,x_n)=\prod_{j=1}^n f_j(x_j)$, with $f_j\in
\mathcal S(\R)$ even  and each $\^f_j(u_j)$ is supported in the
range $|u_j|<s_j$, with $\sum s_j<2$. Then
\begin{equation}\label{comparison eq}
\ave{W_f^{(n)}}_{\mathcal H(2g+1,q)} = A(f)
+O_f\left(\frac{\log g}{g}\right),
\end{equation}
the implied constant independent of the finite field size $q$, and
with $A(f)=A(f_1,...,f_n)$ as in Theorem~\ref{finalthm}.
\end{thm}
To prove  Theorem~\ref{comparison thm} we use a similar approach to
that in \cite{Sarnakappendix, Gao} with some simplifications and
variations arising from our function field setting. In particular
Poisson summation, which is used critically  in
\cite{Sarnakappendix, Gao} is replaced by the functional equation of
the zeta-function $Z_h$.

What is crucial in Theorem~\ref{comparison thm} is that the bound on
the remainder term is uniform in $q$.
Taking the iterated limit  $\lim_{g\to \infty}(\lim_{q\to \infty})$
 of \eqref{comparison eq} and using the Katz-Sarnak result
\eqref{iterated limit} gives our main result on the quadratic
ensemble, as well as a corresponding result for the hyper-elliptic
ensemble:
\begin{cor}\label{main cor}
Let $f\in \mathcal S(\R^n)$ be even in all variables, and assume
that $\^f(u)$ is supported in the region $\sum_{j=1}^n |u_j|<2$.
Then for $q$ fixed,
\begin{equation}\label{hyperelliptic main cor}
 \lim_{g\to \infty} \ave{W_f^{(n)}}_{\mathcal H(2g+1,q)}
=\int_{\R^n} f(x)W^{(n)}_{\USp} (x) dx \;,
\end{equation}
and assuming GRH,
\begin{equation}\label{quadratic main cor}
\lim_{X\to \infty} \ave{W_f^{(n)}}_{\mathcal D(X)}  =\int_{\R^n}
f(x)W^{(n)}_{\USp} (x) dx \;.
\end{equation}
\end{cor}

\begin{proof}
For both \eqref{hyperelliptic main cor} and  \eqref{quadratic main
cor}  we may assume that $f=\prod f_j(x_j)$, with each $\^f_j$ even
and supported on $|u_j|<s_j$ and $\sum s_j<2$, since any $f$
satisfying the conditions of the corollary can be approximated by a
linear combination of functions of this form. Now it follows from
Theorem \ref{comparison thm} and (\ref{iterated limit}) that
(\ref{quadratic main cor}) holds and
$$A(f)=A(f_1,...,f_n)=\int_{\R^n} f(x)W^{(n)}_{\USp} (x) dx.$$ This
is obtained by taking the limit $g\to\infty$ in
Theorem~\ref{comparison thm} and comparing with \eqref{iterated
limit}. Now \eqref{hyperelliptic main cor} follows from \eqref{Gao's
result}.
\end{proof}


\subsection{Further applications}
The method of this paper can in principle be used to compute
statistics of zeros of other families of L-functions, provided a
good function field analogue can be found. For instance, one of us
(A.E.) has given an alternate proof of the result of Rudnick and
Sarnak \cite{rudsar} that  the $n$-level correlation of the Riemann
zeros (that is of a {\em single} L-function) are given by Random
Matrix Theory, by making a comparison with a family of
Artin-Schreier curves \cite{Entin}. Very recently a different
combinatorial proof of the result of \cite{rudsar} was given by
Conrey and Snaith \cite{ConreySnaith}.

\section{Background on function field arithmetic}
We review some elements of the arithmetic of $\fq[x]$. A good
general reference for this material is \cite{Rosen}.
\subsection{Quadratic characters}
Let $P\in\mathbb{F}_{q}[x]$ be a prime polynomial. The quadratic
residue symbol $\left(\frac{f}{P}\right)\in \{\pm1\}$ is defined for
$f$ coprime to $P$ by
$$\left(\frac{f}{P}\right)\equiv f^{\frac{|P|-1}{2}} \pmod{P}.$$
For arbitrary monic $Q\in \mathbb{F}_{q}[x]$ and for $f$ coprime to
$Q$, the Jacobi symbol $(\frac{f}{Q})$ is defined by writing
$Q=\prod P_{j}$ as a product of prime polynomials and setting
$$\left(\frac{f}{Q}\right)=\prod\left(\frac{f}{P_{j}}\right).$$
If $f,Q$ are not coprime we set $(\frac{f}{Q})= 0$.

The law of quadratic reciprocity asserts that for
$A,B\in\mathbb{F}_{q}[x]$ monic polynomials
$$\left(\frac{B}{A}\right)=(-1)^{(\frac{q-1}{2})\deg A\deg B}\left(\frac{A}{B}\right).$$
For $D\in \mathbb{F}_{q}[x]$ a monic polynomial of positive degree
which is not a perfect square, we define the quadratic character
$\chi_{D}$ by
$$\chi_{D}(f)=\left(\frac{D}{f}\right).$$
\subsection{L-functions} For the quadratic character $\chi_{D}$, the corresponding L-function is defined for $|u|<\frac{1}{q}$ by
$$
\mathcal{L}(u,\chi_{D}):=\prod_{P~\pr}(1-\chi_{D}(P)u^{\deg P})^{-1}
=\sum_{\beta\geq 0}A_{D}(\beta)u^{\beta},
$$
with
\begin{equation}\label{def of AD}
A_{D}(\beta):=\sum_{\substack{\deg B=\beta\\B~\monic}}\chi_{D}(B)\;.
\end{equation}
If $D$ is nonsquare of positive degree, then $A_{D}(\beta)=0$ for
$\beta\geq\deg D$ and hence the L-function is in fact a polynomial
of degree at most $\deg D-1.$

Now, assume that $D$ is also square-free. Then
$\mathcal{L}(u,\chi_{D})$ has a trivial zero at $u=1$ if and only if
$\deg D$ is even. Thus
$$\mathcal{L}(u,\chi_{D})=(1-u)^{\lambda}\mathcal{L}^{\ast}(u,\chi_{D}),~~~~\lambda=\left\{ \begin{array}{ll}
1 &\mbox{$\deg D ~\mbox{even},$}\\
0 &\mbox{$\deg D~ \mbox{odd},$}
\end{array}
\right.$$ where $\mathcal{L}^{\ast}(u,\chi_{D})$ is a polynomial of
even degree
$$2\delta=\deg D-1-\lambda$$
satisfying the functional equation
\begin{equation}\label{functional equation}
\mathcal{L}^{\ast}(u,\chi_{D})=(qu^{2})^{\delta}\mathcal{L}^{\ast}(\frac{1}{qu},\chi_{D}).
\end{equation}
We write
$$\mathcal{L}^{\ast}(u,\chi_{D})=\sum_{\beta=0}^{2\delta}A^{\ast}_{D}(\beta)u^{\beta},$$
where $A^{\ast}_{D}(0)=1,$ and the coefficients
$A^{\ast}_{D}(\beta)$ satisfy
\begin{equation}\label{duality_coefficients}
A^{\ast}_{D}(\beta)=q^{\beta-\delta}A^{\ast}_{D}(2\delta-\beta).
\end{equation}
In particular, the leading coefficient is
$A^{\ast}_{D}(2\delta)=q^{\delta}.$

\subsection{The explicit formula}
For $h$  monic, square-free, and of positive degree, the zeta
function of the hyperelliptic curve $y^{2}=h(x)$ is
\begin{equation}
Z_{h}(u)=\frac{\mathcal{L}^{\ast}(u,\chi_{h})}{(1-u)(1-qu)}.
\end{equation}
By the Riemann Hypothesis (proved by Weil) we may write
\begin{equation}\label{def of L*}
\mathcal{L}^{\ast}(u,\chi_{h})=\det (I-u\sqrt{q}\Theta_{h})
\end{equation}
for a unitary $2g\times 2g$ matrix $\Theta_{h}$. Taking a
logarithmic derivative of \eqref{def of L*}
gives
\begin{equation}\label{explicit formula}
-\tr
\Theta_{h}^{n}=\frac{\lambda}{q^{n/2}}+\frac{1}{q^{n/2}}\sum_{\deg
f=n }\Lambda(f)\chi_{h}(f).
\end{equation}
\subsection{The Weil bound}
Assume that $B$ is monic of positive degree and not a perfect
square. Then the Riemann Hypothesis and \eqref{explicit formula}
gives Weil's bound for the character sum over primes:
\begin{equation}\label{weil}
\Big|\sum_{\substack{\deg P=n\\ P \mbox{
}\pr}}\left(\frac{B}{P}\right)\Big|\ll\frac{\deg B}{n}q^{n/2}.
\end{equation}


\subsection{Averaging over $\mathcal{H}_{2g+1}$}
The number of square-free monic polynomials of degree $d$ in
$\mathbb{F}_{q}[x]$ is  $q^d(1-\frac 1q)$ for $d\geq 2$, and in
particular
we have, for $g\geq 1$,
\begin{equation*}
\#\mathcal{H}_{2g+1}=(q-1)q^{2g}.
\end{equation*}

We can execute the averaging over $\mathcal{H}_{2g+1}$ using the
M\"{o}bius function $\mu$ of $\mathbb{F}_{q}[x]$, by recalling that
\begin{equation*}
\sum_{A^{2}|h}\mu(A)= \left\{ \begin{array}{ll}
1 &\mbox{$h$ is square-free,}\\
0 &\mbox{otherwise,}
\end{array}
\right.
\end{equation*}
and hence
\begin{equation}\label{expected value}
\langle
F(h)\rangle=\frac{1}{(q-1)q^{2g}}\sum_{2\alpha+\beta=2g+1}\sum_{\deg
B=\beta}\sum_{\deg A=\alpha}\mu(A)F(A^{2}B),
\end{equation}
the sum being over all monic $A,B.$

For a given polynomial $f\in\mathbb{F}_{q}[x]$ apply \eqref{expected
value} to the quadratic character $h\mapsto \chi_{h}(f)$ to get
\begin{equation}\label{Averaging quadratic characters}
\langle
\chi_{h}(f)\rangle=\frac{1}{(q-1)q^{2g}}\sum_{2\alpha+\beta=2g+1}\sum_{\substack{\deg
A=\alpha\\ \gcd(A,f)=1}}\mu(A)\sum_{\deg
B=\beta}\left(\frac{B}{f}\right).
\end{equation}

\section{A sum of M\"obius values.}
Define
\begin{equation}
\sigma(f,\alpha):=\sum_{\substack{\deg~A=\alpha
\\\gcd(A,f)=1}}\mu(A).
\end{equation}
Note that $\sigma(f,\alpha)$ depends only on the degrees of the
primes dividing $f$, hence we can write for $P_{1},\ldots,P_{n}$
distinct primes of degrees $r_1,\dots, r_n$ respectively:
$\sigma(\prod_{i=1}^{n}P_{i},\alpha)=\sigma(\vec r;\alpha)$.
\begin{lemma}\label{lemm: mobius}
Assume $\min(r_1,\dots, r_n)\geq 2$, then
\begin{equation}
\sigma(\vec r;\alpha)=\left\{ \begin{array}{ll}
1 &\mbox{$\alpha=0$,}\\
-q &\mbox{$\alpha=1$,}\\
0 &\mbox{$2\leq \alpha<\min(r_1,\dots, r_n)$.}
\end{array}
\right.
\end{equation}
In any case we have a bound
\begin{equation}\label{bound on sigma}
|\sigma(\vec r,\alpha)|\leq (q+1)\frac{\alpha^n }{\prod_{j=1}^n r_j}
\;.
\end{equation}
\end{lemma}
\begin{proof}

The lemma 
follows from the identity
\begin{equation}
\sum_{\alpha=0}^\infty \sigma(f,\alpha)X^\alpha =
\sum_{\gcd(A,f)=1}\mu(A)X^{\deg A} = \frac{1-qX}{\prod_{P\mid
f}(1-X^{\deg P})},
\end{equation}
the product being over all prime divisors of $f$.
\end{proof}

For distinct primes $P_1,\dots,P_n$ of degrees $\deg P_j=r_j$, we
define
\begin{equation}\label{def of phi_delta}
\phi_\delta(\vec r) := \sum_{\substack{D\mid \prod P_j\\\deg D\leq
\delta}}\frac{\mu(D)}{q^{\deg D}}.
\end{equation}
As the notation signifies, $\phi_\delta(\vec r)$ depends only on the
degrees of the primes $P_j$, and we can rewrite it as
\begin{equation}
\phi_\delta(\vec r) = \sum_{\substack{I\subset \mathbf n\\
\sigma(I)\leq \delta}} (-1)^{|I|} q^{-\sigma(I)}
\end{equation}
where for a subset $I\subset \mathbf n = \{1,\dots,n\}$ we define
$$\sigma(I):=\sum_{i\in I} r_i$$
and denote by $|I|$ the cardinality of the index set $I$.

Assume now that $\beta$ is odd and $\sum r_j$ is even, and $\sum
r_j>\beta$. Define
\begin{equation}\label{def of Phi_beta}
\Phi_\beta(\vec r):=-q^L \phi_L(\vec r) + (q-1)\sum_{l=0}^{L-1} q^l
\phi_l(\vec r),
\end{equation}
where $2L=\sum r_j-1-\beta$.

\begin{lemma}\label{lemPhi}
Assume  $\beta$ is odd,  $\sum r_j$ is even, and $\beta\leq \sum r_j
-2$. Then
\begin{equation}
\Phi_\beta(\vec r) = -\sum_{\substack{I\subseteq \mathbf n\\
\sigma(I)\leq L}} (-1)^{|I|}.
\end{equation}

\end{lemma}
\begin{proof}
From the definition,
\begin{equation}
\Phi_\beta(\vec r) = -q^L\sum_{\sigma(I)\leq L}(-1)^{|I|}
q^{-\sigma(I)}  + (q-1)\sum_{l=0}^{L-1} q^l\sum_{\sigma(I)\leq
l}(-1)^{|I|} q^{-\sigma(I)}.
\end{equation}
Changing order of summation, we get
\begin{equation}\label{Change order}
\Phi_\beta(\vec r) = \sum_{\substack{I\subseteq \mathbf
n\\\sigma(I)\leq L}} (-1)^{|I|} q^{-\sigma(I)} \left\{
-q^L+(q-1)\sum_{\sigma(I)\leq l\leq L-1}q^l \right\}
\end{equation}
Summing the geometric series gives
\begin{equation}
-q^L+(q-1)\sum_{\sigma(I)\leq l\leq L-1}q^l = -q^{\sigma(I)}
\end{equation}
and inserting in \eqref{Change order} proves the claim.
\end{proof}

\section{Multiple character sums}
Define
\begin{equation}\label{character sums}
S(\beta;\vec r):=\sum_{\substack{ \deg B=\beta\\
B~\monic}} \sum_{\substack{ \deg P_{j}=r_{j}\\P_{i}\neq
P_{j}}}\left(\frac{B}{\prod_{j=1}^{n}P_{j}}\right).
\end{equation}
These sums will play a crucial role in what follows.

By quadratic reciprocity
$$S(\beta;\vec r)=(-1)^{\frac{q-1}{2}\beta(\sum
r_{j})}\sum_{\substack{ \deg P_{j}=r_{j}\\P_{i}\neq
P_{j}}}A_{\prod_{j=1}^{n}P_{j}}(\beta),$$ where the sum is over
distinct primes $P_{j}$ and $A_{F}(\beta)$ given by \eqref{def of
AD} is the coefficient of the L-polynomial
$\mathcal{L}(u,\chi_{F})$. Since the L-function is a polynomial of
degree $\deg F -1$, we have
\begin{lemma}\label{lemm: s equals zero}
If $\beta\geq\sum_{j=1}^{n}r_{j}$ then $S(\beta;\vec r)=0$.
\end{lemma}

\subsection{Duality}
\subsubsection{Duality for $\sum r_{j}$ odd} Assume $\sum r_{j}$ is
odd and $\beta\leq \sum r_{j}-1$. Let $P_{1},\dots, P_n$  be
distinct primes. Then
$\mathcal{L}(u,\chi_{\prod_{j=1}^{n}P_{j}})=\mathcal{L}^{\ast}(u,\chi_{\prod_{j=1}^{n}P_{j}}),$
and so the coefficients
$$A_{\prod_{j=1}^{n}P_{j}}(\beta)=A^{\ast}_{\prod_{j=1}^{n}P_{j}}(\beta)$$
coincide. Therefore, from \eqref{duality_coefficients} we have
$$A_{\prod_{j=1}^{n}P_{j}}(\beta)=A_{\prod_{j=1}^{n}P_{j}}\left(\textstyle\sum
r_{j}-1-\beta\right)q^{\beta-\frac{\sum r_{j}-1}{2}}.$$ Hence if
$\sum r_{j}$ is odd and $\beta\leq \sum r_{j}-1$ then
\begin{equation}\label{duality odd}
S(\beta;\vec r)=q^{\beta-\frac{\sum
r_{j}-1}{2}}S\left(\textstyle\sum r_{j}-1-\beta;\vec r\right).
\end{equation}

\subsubsection{Duality for $\sum r_{j}$ even} Assume $\sum r_{j}$ is
even and $\beta\leq \sum r_{j}-2.$ Let $P_{1},\dots, P_n$ be
distinct primes. Then the equation
$$\mathcal{L}(u,\chi_{\prod_{j=1}^{n}P_{j}})=(1-u)\mathcal{L}^{\ast}(u,\chi_{\prod_{j=1}^{n}P_{j}})$$
implies (here we write $A(\beta)$ for $A_{\prod P_j}(\beta)$)
$$
A(0) =A^{\ast}(0)=1,
$$
$$
A(\textstyle\sum r_{j}-1) =-A^{\ast}(\textstyle\sum r_{j}-2),
$$
$$
A^{\ast}(\beta) =A(\beta)+A(\beta-1)+\cdots+A(0),
$$
and
\begin{equation}\label{4.1 even first}
~~~A(\beta)=A^{\ast}(\beta)-A^{\ast}(\beta-1).
\end{equation}
 From (\ref{duality_coefficients}) we have
\begin{equation}\label{4.1 even second}
A^{\ast}(\beta)=q^{\beta-\frac{\textstyle\sum
r_{j}-2}{2}}A^{\ast}(\textstyle\sum r_{j}-2-\beta).
\end{equation}
Hence
$$A^{\ast}( \sum r_{j}-2)=q^{\frac{ \sum r_{j}-2}{2}},$$
and so
$$A(\textstyle\sum r_{j}-1)=-q^{\frac{ \sum r_{j}-2}{2}}.$$
Therefore, if $\sum r_{j}$ is even then
\begin{equation}\label{s-1 even}
\begin{split}
S( \sum r_{j}-1;\vec r)&=\sum_{\substack{ \deg
P_{j}=r_{j}\\P_{i}\neq P_{l}}}-q^{\frac{\sum r_{j}-2}{2}}
\\
&=-q^{\frac{\sum r_{j}-2}{2}}\pi(r_{1})\cdots
\pi(r_{n})+O\left(q^{\frac{3\sum r_{j}}{2}-\min r_j}\right).
\end{split}
\end{equation}
If $\beta\leq \sum r_{j}-2$ then by (\ref{4.1 even first}) and
(\ref{4.1 even second}) we have
\begin{equation}
A (\beta)= q^{\beta-\frac{\sum r_{j}}{2}} \left(-A (\sum
r_{j}-1-\beta) +(q-1)\sum_{l=0}^{\sum r_{j}-2-\beta}A (l) \right).
\end{equation}
Hence
\begin{equation}\label{duality even}
S(\beta;\vec r)=q^{\beta-\frac{\sum r_{j}}{2}}\left(
                                                              -S(\sum
r_{j}-1-\beta;\vec r)+(q-1)\sum_{l=0}^{\sum
r_{j}-2-\beta}S(l;\vec r) \\
                                                          \right).
\end{equation}

\subsection{Estimates for $S(\beta;\vec r)$}



For the convenience of writing we assume from now on that
$$r_{1}=\min(r_1,\dots, r_n)\;.$$

\begin{lemma}\label{lemm: first estimation s}
\begin{equation*}
S(\beta;\vec r)=\eta_{\beta}q^{\beta/2}\phi_{ \beta/2}(\vec r)
\prod\pi(r_{j}) + O\left(\phi_{\beta/2}(\vec r)\beta^{n}q^{\max(\sum
r_{j}+\frac{\beta}{2}-r_{1},\sum \frac{r_{j}}{2}+\beta)}\right),
\end{equation*}
where $\phi_\delta(\vec r)$ is as defined in \eqref{def of
phi_delta}.
\end{lemma}
\begin{proof}
 We write
\begin{align*}S(\beta;\vec r)=\eta_{\beta}\sum_{\substack{\deg
B=\beta\\B=\Box}}\sum_{\substack {\deg P_{j}=r_{j} \\
P_{i}\neq P_{l}}}\left(\frac{B}{\prod_{j=1}^{n}P_{j}}\right)
+\sum_{\substack{{\deg B=\beta}\\B\neq\Box}}\sum_{\substack{ \deg
P_{j}=r_{j}\\P_{i}\neq
P_{l}}}\left(\frac{B}{\prod_{j=1}^{n}P_{j}}\right),
\end{align*}
where the squares only occur when $\beta$ is even. We write the sum
over squares $B=C^{2}$ as
$$
\sum_{\substack {\deg P_{j}=r_{j}\\P_{i}\neq P_{l}}}\sum_{\deg C=
\frac{\beta}{2}}\left(\frac{C^{2}}{\prod_{j=1}^{n}P_{j}}\right) \;.
$$
The inner sum is the number of C's coprime to
$\prod_{j=1}^{n}P_{j},$ which is
$q^{\frac{\beta}{2}}\phi_{\beta/2}(\vec r)$ (this is seen by the
definition (\ref{def of phi_delta}) of $\phi_{\beta/2}$ and
inclusion-exclusion). Summing over the distinct $P_{j}$ we get that
the sum over square $B$'s is
$$
q^{\frac{\beta}{2}}\pi(r_{1})\cdots\pi(r_{n})\phi_{\frac{\beta}{2}}(\vec
r)+ O\left(\phi_{\beta/2}(\vec r)q^{\sum
r_{j}+\frac{\beta}{2}-r_{1}}\right).
$$
For $B$ not a perfect square, we use Weil's theorem \eqref{weil}.
Hence summing over all non-square $B$ of degree $\beta,$ of which
there are at most $q^{\beta}$, gives
$$\sum_{\substack{{\deg B=\beta}\\B\neq\Box}}\sum_{\substack{ \deg
P_{j}=r_{j}\\P_{i}\neq
P_{l}}}\left(\frac{B}{\prod_{j=1}^{n}P_{j}}\right)\ll
\beta^{n}q^{\sum \frac{r_{j}}{2}+\beta}$$ and with the contribution
of square $B$, this concludes the lemma.
\end{proof}

By using duality, we can improve the estimate of the lemma when
$\beta$ is odd and $\sum r_{j}<2\beta$.

\begin{prop}\label{prop: s eval}
Assume $\beta$ is odd, and $\beta\leq\sum r_{j} -2$. Then
\begin{equation*}
S(\beta;\vec r)=  \eta_{\textstyle\sum r_{j}}q^{\beta}
\Phi_\beta(\vec r) \prod \frac{\pi(r_j)}{q^{r_j/2}}
 +O((\textstyle\sum r_{j})^{n}q^{\textstyle\sum r_{j}})
\end{equation*}
where $\Phi_\beta(\vec r)$ is given in \eqref{def of Phi_beta}.
\end{prop}
\begin{proof}
 Assume $\sum r_{j}$ is odd. Since $\beta\leq\sum
r_{j} -2$ we may use \eqref{duality odd} for $\sum r_{j}$ odd,
$$S(\beta;\vec r)=q^{\beta-\frac{\sum r_{j}-1}{2}}S(\textstyle\sum r_{j}-1-\beta;\vec r)$$
and inserting the inequality of Weil's theorem with $\beta$ replaced
by $\sum r_{j}-1-\beta$ we get
$$S(\textstyle\sum r_{j}-1-\beta;\vec r)\ll(\sum r_{j})^{n}q^{\frac{\sum r_{j}}{2}+(\sum r_{j}-1-\beta)},$$
hence
$$S(\beta;\vec r)\ll q^{\beta-\frac{\sum r_{j}-1}{2}}(\textstyle\sum r_{j})^{n}q^{\frac{\sum r_{j}}{2}+(\sum r_{j}-1-\beta)}\ll
(\textstyle\sum r_{j})^{n}q^{\sum r_{j}}$$ as claimed.
\\Now assume $\sum r_{j}$ is even. Using \eqref{duality even} and Lemma \ref{lemm: first estimation s} we get
\begin{align*}
S(\beta;\vec r)&=q^{\beta-\frac{\sum
r_{j}}{2}}\left(-S(\textstyle\sum r_{j}-1-\beta;\vec r)
+(q-1)\sum_{l=0}^{\sum r_{j}-\beta-2}S(l;\vec r) \right)=\\
&=q^{\beta-\frac{\sum r_{j}}{2}}\pi(r_{1})\cdots\pi(r_{n})\left(
                                                                                       -\eta_{\sum r_{j}-1-\beta}q^{\frac{\sum r_{j}-1-\beta}{2}}
                                                                                       \phi_{\frac{\sum r_{j}-1-\beta}{2}}(\vec r)\right.+\\&+\left.
                                                                                       (q-1)\sum_{l=0}^{\sum r_{j}-\beta-2}\eta_{l}q^{\frac{l}{2}}
                                                                                       \phi_{\frac{l}{2}}(\vec r)
                                                                                     \right)+
                                                                                   \\&+
                                                                                   O \left(
                                                                                      \begin{array}{c}
                                                                                        \phi_{\beta/2}(\vec r)q^{\beta-\frac{\sum r_{j}}{2}+1}\sum_{l=0}^{ \sum r_{j}-\beta-2} l^{n}q^{\max(\frac{\sum r_{j}}{2}+l,\sum r_j+\frac{l}{2}-r_1)}\\
                                                                                      \end{array}
                                                                                    \right).
\end{align*}
The remainder term is $O((\sum r_{j})^{n}q^{\sum r_{j}}).$ For the
main term, we note that $\sum r_{j}-1-\beta$ is even since $\beta$
is odd and $\sum r_{j}$ is even. Denote $2L:=\sum r_{j}-1-\beta$,
then we can write the main term as
\begin{equation*}
q^{\beta}\left(
                                                                                     \begin{array}{c}
                                                                                       -q^{L}\phi_{L}(\vec r)+
                                                                                       (q-1)\sum_{l=0}^{L-1}q^{l}\phi_{l}(\vec r) \\
                                                                                     \end{array}
                                                                                   \right)
                                                                                   \prod \frac{\pi(r_{j})}{q^{r_j/2}}
= q^\beta \Phi_\beta(\vec r) \prod \frac{\pi(r_j)}{q^{r_j/2}}
\end{equation*}
by definition \eqref{def of Phi_beta} of $\Phi_\beta(\vec r)$.
\end{proof}

\section{The $n$-level density}

In the present section we begin the calculation of the $n$-level
density for the hyperelliptic ensemble. First we recall the
definition of $n$-level density. Let $n$ be a natural number and
suppose we are given $n$ real-valued even test function
$f_1,...,f_n\in\sr$ (by $\sr$ we denote the Schwartz space). Let
$$ \^f_k(s)=\int_\R f_k(t)e^{-2\pi ist}dt$$
be the Fourier transforms of $f_k$. We will assume that each $\fh_j$
is supported on the interval $(-s_j,s_j)$ and $\sum s_j<2$. Let
$h\in\hgq$ be a polynomial defining a curve $y^2=h(x)$ with
normalized L-zeros $e^{i\th_j}, j=\pm 1,...,\pm g, \th_{-j}=-\th_j$.
Let
$$\ft_k(t)=\sum_{m\in\Z}f_k\lb 2g\lb\frac{t}{2\pi}+m\rb\rb$$
be the associated periodic test functions (with period $2\pi$). We
denote
$$
W^{(n)}_f(h)=\sum_{\substack{\th_{j_1},...,\th_{j_n}\\1\le |j_k|\le
g\\j_k\neq\pm j_l\mbox{ if }k\neq
l}}\ft_1(\th_{j_1})...\ft_n(\th_{j_n}).
$$
For the rest of the section whenever we use the averaging notation
we mean averaging over $h\in\hgq$ and whenever we use the asymptotic
big-$O$ notation the implicit constant may depend on $n,f_1,...,f_n$
(and other test functions we introduce) but not on $g,q$. The aim of
this section is to prove that \beq\label{wa}
\av{W^{(n)}_f(h)}=A(f_1,...,f_n)+O(\log g/g), \eeq where
$A(f_1,...,f_n)$ is an explicit expression in the $f_i$ and their
Fourier transforms independent of $g,q$.

\subsection{Passage to unrestricted sums}

To express $W^{(n)}_f$
in terms of unrestricted sums over zeros we use a standard
combinatorial sieving method (see \cite{rudsar, Rubinstein, Gao} for
usage of this method in a similar context). First of all since the
$f_i$ are even we may write
$$
W^{(n)}_f=2^n\sum_{\substack{1\le j_1,...,j_n\le g\\
\mbox{dist.}}}\ft_1(\th_{j_1})...\ft_n(\th_{j_n})
$$
(here the summation is over distinct $j_1,...,j_n$).

Denote by $\pin$ the set of partitions of the set ${1,...,n}$. For
two partitions $\uf,\ug\in\pin$ we say that $\uf$ refines $\ug$ and
write $\uf\prec\ug$ if each set appearing in $\ug$ is a union of
sets appearing in $\uf$. We denote $\uo=\{\{1\},...,\{n\}\}\in\pin$.
For any finite set $F$ we denote by $|F|$ its cardinality. For a
partition $\uf=\{F_1,...,F_\nu\}$ we denote $|\uf|=\nu$. Now suppose
we have a function $R:\pin\to\R$ and denote
$C(\uf)=\sum_{\uf\prec\ug}R(\ug)$. The combinatorial M\"{o}bius
inversion formula states that
$R(\uf)=\sum_{\uf\prec\ug}\mu(\uf,\ug)C(\ug)$, where $\mu(\uf,\ug)$
is the M\"{o}bius function for the partially ordered set $\pin$. It
is known that if $\uf=\{F_1,...,F_\nu\}$ then
$$\mu(\uo,\uf)=\prod_{l=1}^\nu (-1)^{|F_l|-1}(|F_l|-1)!$$ (see
\cite[\S 12]{lint}), so we have
$$R(\uo)=\sum_{\uf\in\pin}\prod_{l=1}^{|\uf|} (-1)^{|F_l|-1}(|F_l|-1)!C(\uf)$$ (here $\uf={F_1,...,F_\nu}, \nu=|\uf|$).

Now for $\uf=\{F_1,...,F_\nu\}\in\pin$ take $$C(\uf)=\sum_{1\le
j_1,...,j_\nu\le g}\prod_{l=1}^\nu\prod_{k\in
F_l}\ft_k(\th_{j_l}),$$
$$R(\uf)=\sum_{\substack{1\le j_1,...,j_\nu\le g\\{\mbox{dist.}}}}\prod_{l=1}^\nu\prod_{k\in F_l}\ft_k(\th_{j_l}).$$
It is easy to see that $C(\uf)=\sum_{\uf\prec\ug}R(\ug)$ and so
denoting $$\ut_F(\th)=\prod_{k\in F}\ft_k(\th),$$ we have
\begin{equation*}\label{1}
\begin{split}
W^{(n)}_f&=2^nR(\uo)=2^n\sum_{\uf\in\pin}\prod_{l=1}^{|\uf|} (-1)^{|F_l|-1}(|F_l|-1)!\sum_{1\le j_1,...,j_{|\uf|}\le g} \ut_{F_l}(\th_{j_l})\\
&=2^n\sum_{\uf}\prod_{l=1}^{|\uf|}
(-1)^{|F_l|-1}(|F_l|-1)!\sum_{1\le j\le g}\ut_{F_l}(\th_j).
\end{split}
\end{equation*}
Since the $f_j$ and hence also the $\ut_F$ are even, we may also
rewrite this with a sum over all zeros: \beq\label{e1}
W^{(n)}_f=\sum_{\uf}(-2)^{n-|\uf|}\prod_{l=1}^{|\uf|}(|F_l|-1)!\prod_{l=1}^{|\uf|}\sum_{j=\pm
1,...,\pm g}\ut_{F_l}(\th_j). \eeq

\subsection{Passage to a sum over primes}

Next we will replace the sum over zeros in \rf{e1} with a sum over
primes. For any $f\in\sr$ with compactly supported Fourier transform
we denote \beq\label{defs} \AES(f; h):=\frac{1}{g}\sum_{r=1}^\ity
rq^{-r/2}\hat{f}\lb\frac{r}{2g}\rb\sum_{\substack{\deg P=r
\\  {\mathrm{ prime}} }}\leg{h}{P}. \eeq

\begin{prop}\label{prop1}
Let $f\in\sr$ be a real-valued even function with compactly
supported Fourier transform $\hat{f}$ and let
$\ft(t)=\sum_{m\in\Z}f\lb 2g\lb\frac{t}{2\pi}+m\rb\rb$ be its
associated periodic function. Then for any $h\in\hgq$ with
normalized L-zeros $e^{i\th_j}, j=\pm1,...,\pm g$ we have
$$
\sum_{1\le |j|\le
g}\ft(\th_j)=\fh(0)-\frac{1}{2}f(0)-\AES(f;h)+O(\log g/g)
$$
(the implicit constant may depend on $f$).
\end{prop}

\begin{proof} The Fourier coefficients of $\ft$ are $\widehat{\ft}(r)=\frac{1}{2g}\fh\lb\frac{r}{2g}\rb$, so we have
\beq\label{series}
\ft(t)=\sum_{r\in\Z}\frac{1}{2g}\fh\lb\frac{r}{2g}\rb e^{irt}
=\frac{1}{2g}\fh(0)+\frac{1}{g}\sum_{r=1}^\ity \fh\lb\frac{r}{2g}\rb
e^{irt}. \eeq The explicit formula states that \beq\label{exp}
\sum_{1\le|j|\le g} e^{ir\th_j}=-q^{-r/2}\sum_{\substack{\deg Q=r
\\ {\mathrm{monic}}}}\leg{h}{Q}\Lam(Q), \eeq where $\Lam$ is the
von Mangoldt function. Combining \rf{series} and \rf{exp} we obtain
\beq\label{sum1}\sum_{1\le |j|\le
g}\ft(\th_j)=\fh(0)-\frac{1}{g}\sum_{r=1}^\ity
q^{-r/2}\fh\lb\frac{r}{2g}\rb \sum_{\substack{\deg Q=r\\
{\mathrm{monic}}}}\leg{h}{Q}\Lam(Q). \eeq The contribution to this
sum from prime $Q$ is exactly the term appearing in the statement of
the proposition. Now we consider the contribution of the squares
$Q=P^2$ with $P$ prime, $\deg P=r/2$ (for $r$ even). We use the fact
that $\leg{h}{P^2}$ is 1 unless $P|h$, in which case it is 0. We
denote by $\pi(r)$ the number of monic irreducible polynomials in
$\fq[x]$ of degree $r$. Since $\pi(r)=q^r/r+O(q^{r/2}/r)$, we have
\begin{equation*}
\begin{split}
\sum_{\substack{\deg P=r/2\\
{\mathrm{prime}}}}\leg{h}{P^2}\Lam(P^2)&=\pi(r/2)\frac{r}{2}-\frac{r}{2}\cdot\#\{P
\mbox{ prime}, P|h\}\\
&= q^{r/2}+O(q^{r/4}+\min(g,q^{r/2}))\;,
\end{split}
\end{equation*}
since the number of prime $P|h$ of degree $r/2$ is
$O(\min(g/r+q^{r/2}/r))$. We see that the contribution of these
squares to the sum in \rf{sum1} is
\begin{multline*}\frac{1}{g}\sum_{r=1}^\ity\hat{f}(r/2g)\lb 1+O\lb
q^{-r/4}+\min(gq^{-r/2},1))\rb\rb =\\=2\int_0^\ity\hat{f}(t)\d
t+O(\log g/g)=f(0)+O(\log g/g),\end{multline*} because $\sum_{r>\log
g}gq^{-r/2}=O(1)$. The contribution of higher prime powers $Q=P^k,
k>3$ is $O(1/g)$ because the number of prime $P$ with $\deg P\le
r/3$ is $O(q^{r/3}/r)$.
\end{proof}

\begin{cor}\label{cor1}
\begin{multline*}
W^{(n)}_f=\sum_{\uf} (-2)^{n-|\uf|}\prod_{l=1}^{|\uf|}(|F_l|-1)! \\
\cdot\lb \uh_{F_l}(0)-\frac{1}{2}U_{F_l}(0) -\AES(\uh_{F_l}; h) +
O(\frac{\log g}{g})\rb,
\end{multline*}
where $U_F(t)=\prod_{k\in F} f_k(t)$, $\uh_F$ is its Fourier
transform.\end{cor}

\begin{proof}
This follows from \rf{e1} and Proposition \ref{prop1}. Note that
$\ut_F(t)$ is the associated periodic function of $U_F$.
\end{proof}

Now let $u_1,...,u_k\in\sr$ with $k\le n$ be real-valued even
functions with Fourier transforms $\hat{u}_l$. We denote
$$M(u_1,...,u_k)=\av{\prod_{l=1}^k \AES(u_l;h)},$$ ($\AES(u_l;h)$ is defined by \rf{defs}).
In the next subsection we will prove that if $\hat{u}_l$ is
supported in $(-\del_l,\del_l)$ and $\sum_{l=1}^k\del_l<2$ then
\beq\label{mb} M(u_1,...,u_k)=B(u_1,...,u_k)+O(\log g/g), \eeq where
$B(u_1,...,u_k)$ is an explicit expression in the $u_l$ and their
Fourier transforms which is independent of $g,q$.

\begin{prop}\label{arg}
Suppose that \rf{mb} holds under the appropriate conditions on the
supports of $\hat{u}_l$. Then
$$
\av{W^{(n)}_f }=A(f_1,...,f_n)+O(\log g/g)
$$
holds with
\begin{multline*}
A(f_1,...,f_n)=\sum_{\uf} (-2)^{n-|\uf|}\prod_{l=1}^{|\uf|}(|F_l|-1)!\sum_{S\ss\{1,...,l\}}\lb\prod_{l\in S^c}\uh_{F_l}(0)\rb\cdot\\
\cdot \sum_{S_2\ss S}(-1/2)^{|S_2^c|}\lb\prod_{l\in
S_2^c}U_{F_l}(0)\rb (-1)^{|S_2|} B(U_{l_1},...,U_{l_{|S_2|}}),
\end{multline*}
where the first summation is over all partitions
$\uf=\{F_1,...,F_{|\uf|}\}\in\pin$, the second is over all subsets
$S\in\{1,...,l\}$, $S^c$ denotes the complement of $S$ in
$\{1,...,l\}$, the third summation is over all subsets
$S_2=\{l_1,...,l_{|S_2|}\}\ss S$, and $S_2^c=S\setminus S_2$.
\end{prop}

\begin{proof} First we note that if we could ignore the $O(\log g/g)$ terms in Corollary \ref{cor1} then the Proposition would follow at once by expanding
the product, averaging and using \rf{mb}. Here we use the fact that
$\uh_{l_j}$ is supported on the interval $(-\del_j,\del_j)$ where
$\del_j=\sum_{k\in F_{l_j}}s_k$ (recall that $\fh_k$ is supported on
$(-s_k,s_k)$), because the Fourier transform takes products to
convolutions, so we have $\sum_{j=1}^{|S_2|}\del_j\le\sum_{k=1}^n
s_k<2$, which makes \rf{mb} applicable.

To deal with the error terms $O(\log g/g)$ we prove by induction on
$m$ that for any even real-valued $u_1,...,u_m\in\sr$, with each
$\hat{u}_l$ supported on $(-\del_l,\del_l)$ and $\sum\del_l<2$, we
have \begin{multline*}\av{\prod_{l=1}^m
(u_l(0)-\frac{1}{2}\hat{u}_l(0)-\AES(u;h)+O(\log g/g))}=\\=
\av{\prod_{l=1}^m (u_l(0)-\frac{1}{2}\hat{u}_l(0)-\AES(u;h))} +
O(\log g/g)\end{multline*} (see \cite{Rubinstein}, Lemma 2 for a
similar argument). Assuming by induction that this holds for $m-1$
it is enough to show that \beq\label{cond1}\av{O(\log
g/g)\cdot\prod_{l=1}^{m-1}(u_l(0)-\frac{1}{2}\hat{u}_l(0)-\AES(u_l;h))}=O(\log
g/g).\eeq But
$$u_l(0)-\frac{1}{2}\hat{u}_l(0)-\AES(u_l;h)=\sum_{1\le |j|\le g}
\tilde{u}_l(\th_j)+O(\log g/g),$$ (by Proposition \ref{prop1}, here
$e^{i\th_j}$ are the normalized L-zeros corresponding to $h$ and
$\tilde{u}_l$ is the periodic function associated with $u_l$), so by
induction it is enough to show that
$$\av{O(\log g/g)\cdot\prod_{l=1}^{m-1}\sum_{1\le |j|\le g} \tilde{u}_l(\th_j)}=O(\log g/g).$$
For this we may replace each $u_l$ with an even real-valued function
$v_l\in\sr$ s.t. $v_l(t)>|u_l(t)|$ for all $t\in\R$ and each
$\hat{v}_l$ supported on $(-\del_l,\del_l)$. That such functions
always exist is shown in \cite{Rubinstein}, proof of Lemma 2. Now
applying \ref{mb} and using the induction hypothesis we see that
$$\av{\prod_{l=1}^{m-1}\sum_{1\le |j|\le g}
\tilde{v}_l(\th_j)}=O(1),$$ which implies \rf{cond1}.\end{proof}

\subsection{Evaluation of $M(u_1,...,u_m)$: reduction to sums over distinct primes}\label{Sec:M}

In the rest of this section we evaluate
$$M(u_1,...,u_m)=\av{\prod_{l=1}^m \AES(u_l;h)}$$ up to $O(\log
g/g)$ for even real-valued $u_k\in\sr$ s.t. $\hat{u}_k$ is supported
on $(-\del_k,\del_k)$ with $\sum_{k=1}^m\del_k<2$. We want to derive
a result of the form \rf{mb}, so we assume by induction that it
already holds for all $m'<m$. We denote $\mm=\{1,...,m\}$.

Let $F$ be a subset of $\mm$. Denote
$$C(F)=C(F;h)=\frac{1}{g^{|F|}}\sum_{r=1}^\ity q^{-|F|r/2}r^{|F|}\prod_{k\in F}\hat{u}_k\lb\frac{r}{2g}\rb
\sum_{\substack{\deg P=r \\
{\mathrm{prime}}}}\leg{h}{P^{|F|}}.$$ For a partition
$\uf=\{F_1,...,F_\nu\}$ of a set $S\ss\mm$ we denote
$C(\uf)=\prod_{l=1}^\nu C(F_l)$. For two elements $i,j\in S$ we say
that $i\sim_{\uf} j$ if they lie in the same element of $\uf$. We
have
\beq\label{cuf}C(\uf)=\frac{1}{g^{|S|}}\sum_{r_1,...,r_{|S|}=1}^\ity
 \lb\prod_{k\in S}\hat{u}_k\lb\frac{r_k}{2g}\rb \frac{r_k}{q^{r_k/2}}\rb
\sum_{\substack{ P_1,...,P_{|S|}\\
{\mathrm{prime}} \\ {P_i=P_j\mbox{ if } i\sim_\uf j}}}\leg{h}{P_1
\cdots P_{|S|}}.\eeq Define also
\begin{equation}\label{def of R(S)}
R(\uf)=\frac{1}{g^{|S|}}\sum_{r_1,...,r_{|S|}=1}^\ity \lb\prod_{k\in
S}\hat{u}_k\lb\frac{r_k}{2g}\rb \frac{r_k}{q^{r_k/2}}\rb
\sum_{\substack{ P_1,...,P_{|S|}\\
{\mathrm{prime}}\\ {P_i=P_j\mbox{ iff } i\sim_\uf
j}}}\leg{h}{P_1...P_{|S|}}
\end{equation}
(same expression except that the "if" is replaced with an "iff"). We
have \beq\label{cr}C(\uf)=\sum_{\uf\prec\ug}R(\ug), \mbox{ }
R(\uf)=\sum_{\uf\prec\ug}\mu(\uf,\ug)C(\ug).\eeq

\begin{prop}\label{prop2} Let $F$ be a subset of $\mm$. If $F=\{a,b\}$ consists of two (distinct) elements then
$$C(F)=2\int_\R \uu_a(t)\uu_b(t)|t|\d t+O(\log g/g).$$
If $|F|>2$ then $C(F)=O(1/g)$.\end{prop}

\begin{proof} First suppose $F=\{a,b\}$. Then
$$C(F)=\frac{1}{g^2}\sum_{r=1}^\ity\uu_a\lb\frac{r}{2g}\rb\uu_b\lb\frac{r}{2g}\rb r^2q^{-r}\sum_{\substack{\deg P=r\\ {\mathrm{prime}}}}\leg{h}{P^2}.$$
As in the proof of Proposition \ref{prop1} we see that
$$rq^{-r}\sum_{\substack{\deg
P=r\\
{\mathrm{prime}}}}\leg{h}{P^2}=1+O(q^{-r/2}+\min(1,gq^{-r/2}))$$ and
so
\begin{multline*}C(F)=4\sum_{r=1}^\ity \uu_a\lb\frac{r}{2g}\rb\uu_b\lb\frac{r}{2g}\rb\frac{r}{2g}\cdot\frac{1}{2g}+O(\log g/g)=\\
=2\int_\R \uu_a(t)\uu_b(t)|t|\d t + O(\log g/g).\end{multline*}

Now suppose that $|F|=e\ge 3$. Then
$$C(F)\ll\sum_{r=1}^\ity\frac{1}{g^e}q^{(1-e/2)r}r^{e-1}=O(g^{-e}).$$\end{proof}

For any subset $S\ss\mm$ denote $$\uo_S=\{\{k\}|k\in S\},\mbox{ }
\uo=\{\{1\},...,\{m\}\}.$$

\begin{lem}\label{upper} For any proper subset $S\ss\mm$ there is a function $X:\hgq\to\R_{\ge 0}$ s.t. $X(h)\ge |C(\uo_S;h)|$ for all $h$
and $\av{X(h)}=O(1)$.\end{lem}

\begin{proof} Since $C(\uo_S)=\prod_{i\in S} \AES(u_i;h)$ and by Proposition \ref{prop1} we can write
$$C(\uo_S)=\prod_{i\in S}\sum_\th u_i(\th)+\sum_{T\subsetneq S}C(\uo_T)\cdot O(1),$$ where the sum is over the normalized L-zeros corresponding to $h$.
We may assume by induction that $C(\uo_T),T\subsetneq S$ (and
therefore also $C(\uo_T)\cdot O(1)$) satisfy the assertion, so it is
enough to prove it for $\prod_{i\in S}\sum_\th u_i(\th)$. For this
we may replace the $u_i$ with $v_i\ge|u_i|$ s.t. $\hat{v}_i$ is
supported on $(-s_i,s_i)$, as we did in the proof of Proposition
\ref{arg}, so that $$\prod_{i\in S}\sum_\th v_i(\th)\ge
\mid\prod_{i\in S}\sum_\th u_i(\th)\mid$$ for all $h$. Now since $S$
is proper we can apply our induction hypothesis.
\end{proof}

If $S=\{k_1,...,k_\nu\}$ we have
$M(u_{k_1},...,u_{k_\nu})=\av{C(\uo_S)}$. If $S$ is a proper subset
of $\mm$ we may assume by induction that
$$M(u_{k_1},...,u_{k_\nu})=B(u_{k_1},...,u_{k_\nu})+O(\log g/g),$$
where $B(u_{k_1},...,u_{k_\nu})$ depends only on
$u_{k_1},...,u_{k_\nu}$.

\begin{prop}\label{arg2} Let $\uf\in\Pi_m$ be a partition.
let $\{a_i,b_i\},i=1,...,\mu$ be the two-element sets appearing in
$\uf$ and $\{c_i\}$, $i=1,...,\ka$, the one-element sets appearing
in $\uf$. Assume that at least one element of $F\in\uf$ satisfies
$|F|>1$. If some $F\in\uf$ has more than two elements then
$\av{C(\uf)}=O(1/g)$. Otherwise
$$\av{C(\uf)}=B(u_{k_1},...,u_{k_\ka})2^\mu\prod_{i=1}^\mu\int_\R \uu_{a_i}(t)\uu_{b_i}(t)|t|\d t + O(\log g/g).$$\end{prop}

\begin{proof} Denote $S=\{c_i,1\le i\le\ka\}$. We have $\av{C(\uo_S)}=B(u_{c_1},...,u_{c_\mu})+O(\log g/g)$. Denote $\ug=\uf\setminus\uo_S$ (these are
exactly the sets with more than one element in $\uf$). If at least
one set in $\uf$ has more than two elements then by Proposition
\ref{prop2} $C(\ug)=O(1/g)$. Otherwise
$$
C(\ug)=2^\mu\prod_{i=1}^\mu\int_\R \uu_{a_i}(t)\uu_{b_i}(t)|t|\d
t+O(\log g/g).
$$
In both cases we want to show that if we multiply the corresponding
error by $C(\uo_S)$ and average we get the same order of error. This
follows from Lemma \ref{upper}, since we can bound $C(\uo_S)$ by a
suitable $X(h)$.
\end{proof}

In the next subsection we will show that
\beq\label{rd}\av{R(\uo)}=D(u_1,...,u_m)+O(1/g),\eeq where
$D(u_1,...,u_m)$ is an explicit expression depending only on
$u_1,...,u_m$. For a subset $S=\{k_1,...,k_\nu\}\in\mm$ we denote
$D(S)=D(u_{k_1},...,u_{k_\nu})$. Assuming \rf{rd} we prove the
following:
\begin{prop}\label{pairup}
$M=B+O(\log g/g)$, where
\begin{multline*}
B(u_1,...,u_m)=2^{m/2}\sum_{\mathrm{pair}\mbox{ }\mathrm{up}\mbox{ }\mm}\prod_{i=1}^{m/2}\int_\R \uu_{a_i}(t)\uu_{b_i}(t)|t|\d t+\\
+\sum_{S\subsetneq\mm}2^{|S|/2}\sum_{\mathrm{pair}\mbox{
}\mathrm{up}\mbox{ }S} \prod_{i=1}^{|S|/2}\int_\R
\uu_{a_i}(t)\uu_{b_i}(t)|t|\d t \cdot D(S^c).
\end{multline*}
Here the first sum is over all perfect pairings of $\mm$, i.e.
partitions of $\mm$ of the form $$\{\{a_i,b_i\},i=1,...,m/2\},
a_i\neq b_i$$ (if $m$ is odd the sum is empty), the second sum is
over the proper subsets $S\ss\mm$, the third sum is like the first
only for $S$ and $S^c=\{1,...,m\}\setminus S$.\end{prop}

\begin{proof}
We have \beq\label{eq1}
M(u_1,...,u_m)=C(\uo)=\sum_{\uf\in\Pi_m}R(\uf). \eeq First we note
that if $\uf$ is a partition of $S\ss\mm$ that has an element
$F\in\uf$ with $|F|>2$ then $\av{R(\uf)}=O(1/g)$. This follows from
\rf{cr} and Proposition \ref{arg2}. Next we observe that if $|F|=2$
for all the elements $F\in\uf$ then $\av{R(\uf)}=\av{C(\uf)}$,
because of \rf{cr}, the fact that $\mu(\uf,\uf)=1$ and because every
proper $\ug\succ\uf$ has $F\in\ug$ with more than two elements. More
generally, if $a_1,...,a_\mu,b_1,...,b_\mu,c_1,...,c_\ka\in\mm$ are
distinct elements and
$$\uf=\{\{a_1,b_1\},...,\{a_\mu,b_\mu\},\{c_1\},...,\{c_\ka\}\}, \quad S=\{\{c_1\},...,\{c_\ka\}\},$$
then
$$\lb\prod_{i=1}^\mu C(\{a_i,b_i\})\rb R(\uo_S) = R(\uf) + \sum_{\ug\succ\uf}R(\ug),$$
where the sum is only over those proper $\ug\succ\uf$ which leave
the elements of $S$ in different sets. In particular each $\ug$
contains a set with more than two elements. We conclude that
$$\av{R(\uf)}=\av{\lb\prod_{i=1}^\mu C(\{a_i,b_i\})\rb
R(\uo_S)}+O(1/g).$$ Now the proof of Proposition \ref{arg2} can be
imitated to show that
$$\av{R(\uf)}=D(S)\cdot 2^\mu\prod_{i=1}^\mu\int_\R \uu_{a_i}(t)\uu_{b_i}(t)|t|\d t + O(\log g/g)$$
(the required bound $|R(\uo_S)|\le X(h)$ with $\av{X(h)}=O(1)$
follows from \rf{cr}, Lemma \ref{upper} and Proposition \ref{arg2}).
Combining this with \rf{eq1} gives the assertion.\end{proof}

It remains for us to evaluate $\av{R(\uo)}$ and show that
$$\av{R(\uo)}=D(u_1,...,u_m)+O(\log g/g),$$ where $D(u_1,...,u_m)$ is
an explicit expression depending on $u_1,...,u_m$ (and find this
expression). We recall that (compare \eqref{def of R(S)})
\beq\label{r} R(\uo)=\frac{1}{g^m}\sum_{r_1,...,r_m=1}^\ity
\lb\prod_{i=1}^m \uu_i\lb\frac{r_i}{2g}\rb \frac{r_i}{q^{r_i/2}}\rb
\sum_{\substack{\deg P_i=r_i
\\ \mathrm{distinct}\,  \mathrm{primes}}}\leg{h}{P_1\cdots
P_m}. \eeq

To evaluate the average of this expression we need to know, for a
particular tuple $(r_1,...,r_m)$, the average of
\begin{equation}\label{def of P}
\mathcal P(r_1,...,r_m):=\lb\prod_{i=1}^m \frac{r_i}{q^{r_i/2}}\rb
\sum_{\substack{\deg P_i=r_i
\\ {\mathrm{distinct \,primes}}}}\leg{h}{P_1 \cdots
P_m}.
\end{equation}
We will compute this average in the following section.

\section{Estimation of $\ave{\mathcal P(\vec r)}$  }
In this section we focus  on the contribution $\mathcal P(\vec r)$
of different primes defined by \eqref{def of P}. We use
\eqref{Averaging quadratic characters} and the explicit formula of
\eqref{explicit formula} for the mean value of $\mathcal{P}(\vec
r)$:

\begin{equation} \label{AveP}
\begin{split}
\langle\mathcal{P}(\vec r)\rangle   &= \frac{\prod_{j=1}^{m} r_{j}
}{q^{\frac{\sum r_{j}}{2}+2g}(q-1)} \sum_{\substack{ \deg
P_{j}=r_{j}\\P_{i}\neq
P_{l}}}\sum_{2\alpha+\beta=2g+1}\sum_{\substack{\deg A=\alpha\\
\gcd(A,P_{j})=1}}\mu(A)\sum_{\deg
B=\beta}\left(\frac{B}{\prod_{j=1}^{m}P_{j}}\right) \\
&= \frac{\prod_{j=1}^{m}r_{j}}{q^{\frac{\sum r_{j}}{2}+2g}(q-1)}
\sum_{0\leq\alpha\leq g}\sigma(\vec r;\alpha)S(2g+1-2\alpha;\vec r).
\end{split}
\end{equation}

\begin{prop}\label{prop:distinct prime first eval}
Assume that $\sum r_j<(1-\delta)4g$.  Then
\begin{equation*}
\langle\mathcal{P}(\vec r)\rangle
=\frac{\prod_{j=1}^{m}r_{j}}{q^{\frac{\sum r_{j}}{2}+2g}(q-1)}
(S(2g+1;\vec r) -qS(2g-1;\vec r)) + O(q^{-\delta g}+q^{-r_1/2}).
\end{equation*}
\end{prop}
\begin{proof}
  It suffices to show that the terms with $\alpha\geq
2$ contributes $ O(q^{-\delta g}+q^{-r_1/2})$.

Note that $\sigma(\vec r,\alpha)=0$ unless $\alpha\geq r_1:=\min
r_j$ by Lemma \ref{lemm: mobius}. Thus it suffices to take
$\alpha\geq r_{1}$. Recall that in any case,
\begin{equation}\label{Bd on sigma}
|\sigma(\vec r;\alpha)|\leq (q+1)\frac{\alpha^m}{\prod r_j}.
\end{equation}


If $\sum r_j \leq 2g-3$ then $S(2g+1-2\alpha,\vec r)=0$ for
$\alpha\geq 2$ by Lemma~\ref{lemm: s equals zero}. Thus we may
assume that $\sum r_j\geq 2g-2$.

We first assume that $\sum r_j\geq 2g-1$  so that for $\alpha\geq
2$, we have $\beta\leq \sum r_j-2$. Using duality, we obtained a
bound for the sums $S(\beta;\vec r)$ in Proposition~\ref{prop: s
eval} which implies that if
$\beta\leq\sum r_j-2$, 
\begin{equation} \label{duality bound}
| S(\beta;\vec r) | \ll q^{\beta-\frac 12 \sum r_j}\prod \pi(r_j)
+(\textstyle\sum r_j)^m q^{\sum r_j} \;.
\end{equation}


We insert \eqref{duality bound} into \eqref{AveP} and first bound
the contributions of the second term on the RHS of \eqref{duality
bound}, call it $II$, namely of $ (\sum r_j)^m q^{\sum r_j} $.
Inserting \eqref{Bd on sigma} and using $\sum r_j<(1-\delta)4g$ we
get
\begin{equation}
\begin{split}
II &\ll \frac 1{q^{2g+1}}\prod \frac{r_j}{q^{r_j/2}}\sum_{r_1\leq
\alpha\leq g} q\frac{\alpha^m}{\prod r_j}(\textstyle\sum r_j)^m q^{\sum r_j}\\
&\ll q^{\frac 12\sum r_j-2g} g^{2m+1} \ll q^{-\delta g}  \;.
\end{split}
\end{equation}

Now for the contribution of the first term on the RHS of
\eqref{duality bound}, call it $I$, which we can bound by
\begin{equation}
\begin{split}
I&\ll \frac 1{q^{2g+1}}\prod \frac{r_j}{q^{r_j/2}}\sum_{r_1 \leq
\alpha\leq g} q\frac{\alpha^m}{\prod r_j}q^{2g+1-2\alpha-\frac 12
\sum
r_j}\prod \pi(r_j) \\
&\ll \frac q{\prod r_j} \sum_{ \alpha \geq r_1  }
\frac{\alpha^m}{q^{2\alpha}}   \ll \frac q{\prod
r_j}\frac{r_1^m}{q^{2 r_1}} \ll q^{-r_1/2}
\end{split}
\end{equation}
on using the bound
$\sum_{\alpha\geq r} \alpha^m z^\alpha \ll_m r^m z^r$,  ($|z|\leq
\frac 14$, $r\geq 1$, $m\geq 1$),
giving our claim when $\sum r_j\geq 2g-1$.

It remains to deal with the case $\sum r_j=2g-2$ and $\alpha=r_1=2$,
where we need to bound the contribution to $\ave{\mathcal P(\vec
r)}$ of
\begin{equation}\label{extra term}
\frac 1{(q-1)q^{2g}} \prod \frac{r_j}{q^{r_j/2}} \sigma(\vec r,2)
S(2g-3,\vec r) \ll \frac 1{q^{2g+\frac 12 \sum r_j}}|S(2g-3,\vec
r)|.
\end{equation}
By \eqref{s-1 even}, if $\sum r_j-1=2g-3$ then
$$
|S(2g-3,\vec r)|\ll q^{3g-3} \left(\frac 1{q\prod r_j}+ \frac
1{q^{r_1}}\right)
$$
and hence
$$
\eqref{extra term} \ll \frac 1{q^2}\left(\frac 1{q\prod r_j}+ \frac
1{q^{r_1}}\right)\;,
$$
and since $\prod r_j \geq \max r_j \geq  \sum r_j/m \geq g/m$, we
recover the proposition in this case as well.
\end{proof}

We now compute $\ave{\mathcal P(\vec r)}$.  For a subset of indices
$I\subseteq \mathbf m=\{1,\dots,m\}$ we denote its complement by
$I^c$. Each subset $I\subseteq \mathbf m$ defines a hyperplane
\begin{equation}\label{exceptional}
\sigma(I^c)-\sigma(I)=2g \;.
\end{equation}
We will call these $2^m$ hyperplanes "exceptional".
\begin{prop}\label{prop:computing aveP}
Assume $\sum r_j<(1-\delta)4g$.

i)  If $\sum_{j=1}^{m}r_{j}>2g+2$ and $\sum_{j=1}^{m}r_{j}$ is even,
then  away from the exceptional hyperplanes \eqref{exceptional} we
have
\begin{equation}\label{P generic}
\ave{\mathcal P(\vec r)} = -\sum_{\sigma(I)< \sigma(I^c)-2g}
(-1)^{|I|} + O(q^{-\delta g} +q^{-r_1/2}).
\end{equation}

ii) If $\sum_{j=1}^m r_j = 2g,   2g+2$  or if $\sum_{j=1}^m
r_j>2g+2$ and \eqref{exceptional}  holds for some  $I\subset \mathbf
m$, then
\begin{equation}\label{P exceptional}
|\ave{\mathcal{P}(\vec r)}|  =O(1).
\end{equation}

iii)  If $\sum_{j=1}^{m}r_{j}<2g$  or if $\sum_{j=1}^{m}r_{j}>2g $
 and   $\sum_{j=1}^{m}r_{j}$ is odd, then
$$
|\ave{\mathcal{P}(\vec r)}|   \ll  q^{-\delta g} +q^{-r_1/2}.
$$

\end{prop}

\begin{proof}
\bigskip
\noindent{\bf The case $\sum_{j=1}^{m}r_{j}<2g$:} We use Proposition
\ref{prop:distinct prime first eval} and note that in this case
$S(2g\pm1;\vec r)=0$ by Lemma \ref{lemm: s equals zero}. Hence
$$\langle\mathcal{P}(\vec r)\rangle=O(q^{-\delta g}  +q^{-r_1/2}).$$

\bigskip
\noindent{\bf The case $\sum_{j=1}^{m}r_{j}=2g$:} For $\sum
r_{j}=2g$ we have $S(2g+1;\vec r)=0$ by Lemma \ref{lemm: s equals
zero}. Thus by Proposition \ref{prop:distinct prime first eval}
$$
\langle\mathcal{P}(\vec r)\rangle=
-\frac{\prod_{j=1}^{m}r_{j}}{q^{\sum_{j=1}^{m}r_{j}/2+2g}(q-1)}qS(2g-1;\vec
r)+ O(q^{-\delta g}  +q^{-r_1/2}).
$$
By \eqref{s-1 even} and using $\sum r_{j}=2g$, we have

\begin{multline*}
\langle\mathcal{P}(\vec r)\rangle = \frac{\prod r_j}{q^{\frac 12
\sum r_j+2g}(q-1)}q \cdot q^{\frac 12 \sum r_j-1} \#\{\deg P_j=r_j,
P_i\neq P_j\} +\\+ O(q^{-\delta g}  +q^{-r_1/2})=\\
= \frac{1}{q-1}+ O(q^{-\delta g}  +q^{-r_1/2}) = O(1).
\end{multline*}

\bigskip
\noindent{\bf The case $\sum_{j=1}^{m}r_{j}=2g+1$:} We have
$S(2g+1;\vec r)=0$ by Lemma \ref{lemm: s equals zero}. Thus by
Proposition \ref{prop:distinct prime first eval}
$$
\langle\mathcal{P}(\vec r)\rangle
=-\frac{\prod_{j=1}^{m}r_{j}}{q^{\sum_{j=1}^{m}r_{j}/2+2g}(q-1)}qS(2g-1;\vec
r) + O(q^{-\delta g}  +q^{-r_1/2})
$$
By Proposition \ref{prop: s eval}, and using $\sum r_{j}=2g+1$ in
\eqref{s-1 even}, we have
$$
\langle\mathcal{P}(\vec r)\rangle \ll \frac{g^{2m}}{q^{g-\frac 12}}
+ q^{-\delta g}  +q^{-r_1/2}
  = O(q^{-\delta g}  +q^{-r_1/2}) \;.
$$

\bigskip
\noindent{\bf The case $\sum_{j=1}^{m}r_{j}=2g+2$:} By Proposition
\ref{prop:distinct prime first eval}
\begin{align*}
\langle\mathcal{P}(\vec
r)\rangle=\frac{\prod_{j=1}^{m}r_{j}}{q^{3g+1}(q-1)} (S(2g+1;\vec
r)-qS(2g-1;\vec r)) +O(q^{-\delta g}  +q^{-r_1/2}) \;.
\end{align*}
Using \eqref{s-1 even} 
we have
$$
S(2g+1;\vec r)=\frac{-q^{3g+2}}{\prod r_{j}}\Big(1+O(q^{-r_1/2})
\Big) \;,
$$
and by Proposition \ref{prop: s eval},
$$
S(2g-1;\vec r)=O\left(\frac{q^{3g}}{\prod r_{j}}\right)\;.
$$
Hence  $\ave{\mathcal{P}(\vec r)}=O(1)$.

\bigskip
\noindent{\bf The case $\sum_{j=1}^{m}r_{j}>2g+2$:} In this case
$\beta=2g\pm1$ satisfies $\beta\leq \sum_{j=1}^{m}r_{j}-2$, hence we
may use Proposition \ref{prop: s eval} which gives that for $\sum
r_j$ even, $\beta$ odd, and $\sum r_j-2\geq\beta$,
\begin{equation}\label{S in terms of Phi}
S(\beta,\vec r) =    q^\beta\Phi_\beta(\vec r) \prod
\frac{\pi(r_j)}{q^{r_j/2}} + O(q^{-\delta g} + q^{-r_1/2}).
\end{equation}
If $\sum r_j$ is odd then there is no main term.

We now insert \eqref{S in terms of Phi} and Lemma~\ref{lemPhi} in
the computation of $\ave{\mathcal P(\vec r)}$ to get that, up to a
remainder term of $O(q^{-\delta g} + q^{-r_1/2})$, we have that if
$\sum r_j>2g$, $\sum r_j$ even then
\begin{equation}\label{intermediate aveP}
\begin{split}
\ave{\mathcal P(\vec r)} &\sim \frac 1{q^{2g}(q-1)}\prod_j
\frac{r_j}{q^{r_j/2}} (S(2g+1,\vec r)-q S(2g-1,\vec r) )\\
&\sim  \Phi_{2g+1}(\vec r) +\frac{\Phi_{2g+1}(\vec r) -
\Phi_{2g-1}(\vec r)}{q-1}\\
&= -\sum_{\sigma(I)\leq L_+} (-1)^{|I|} + \frac{\Phi_{2g+1}(\vec r)
- \Phi_{2g-1}(\vec r)}{q-1},
\end{split}
\end{equation}
where $2L_+=\sum r_j-1-(2g+1)$.

We have $\sigma(I) + \sigma(I^c)=\sum r_j$ and hence the condition
$\sigma(I) \leq L_+$ becomes $\sigma(I)-\sigma(I^c) \leq -(2g+2)$,
and since $\sum r_j=\sigma(I)+\sigma(I^c)$ is even, so is
$\sigma(I)-\sigma(I^c)$ and thus this condition is equivalent to
$$\sigma(I)-\sigma(I^c) < -2g \;.$$
Moreover,
$$
\Phi_{2g+1}(\vec r) - \Phi_{2g-1}(\vec r) =
\sum_{\sigma(I^c)-\sigma(I)=2g} (-1)^{|I|}
$$
and so the second term in \eqref{intermediate aveP} vanishes off the
exceptional hyperplanes \eqref{exceptional}. Thus we have shown
\eqref{P generic} and \eqref{P exceptional}.
\end{proof}

\section{Conclusion}

Now we are ready to prove

\begin{prop}\label{main}
The mean value of $R(\uo)$ is
\begin{equation*}
\av{R(\uo)}= -2^{m-1}\sum_{I\ss\mm}(-1)^{|I|}
\int\limits_{\substack{t_1,\dots,t_m\geq 0\\ \sum t_i\ge 1\\
\sum_{i\in I}t_i\leq \sum_{i\in I^c}t_i-1}}
\prod_{i=1}^m(\uu_i(t_i)\d t_i) + O(1/g)\;.
\end{equation*}
\end{prop}

\begin{proof} We average \rf{r} over $\hgq$ substituting the values provided by Proposition~\ref{prop:computing aveP}.
First let us ignore the errors and examine the contribution of the
main terms. We get that the main term in $\av{R(\uo)}$ is
\begin{multline*}
-2^m\sum_{I\ss\mm}(-1)^{|I|}\sum_{\substack{r_1,...,r_m\geq 1\\
\sum r_i>2g+2\mbox{ }\mathrm{even}\\
\sum_{i\in I}r_i<\sum_{i\in I^c}r_i-2g}}
\prod_{i=1}^m\lb\uu_i\lb\frac{r_i}{2g}\rb\cdot\frac{1}{2g}\rb \\
=-2^{m-1}\sum_{I\ss\mm}(-1)^{|I|}\int\limits_{\substack{\R_{\ge 0}^m\\
\sum t_i\ge 1\\ \sum_{i\in I}t_i\le \sum_{i\in
I^c}t_i-1}}\prod_{i=1}^m(\uu_i(t_i)\d t_i) + O(1/g)
\end{multline*}
using an approximation of the integral by a Riemann sum with step
$1/2g$, with the restriction that $\sum r_j$ is even providing a
factor of $1/2$.
This is the main term in the assertion.

Now we consider the various error terms. Due to the condition on the
supports of $\uu_i$ we only need to consider $\sum r_i<
(1-\delta)4g$ for some fixed $\delta>0$. For the   error term  of
the form $O(q^{-\delta g})$, we use that the number of suitable
tuples $r_i$ is $O(g^m)$, so the total contribution of these errors
is $O(q^{-\delta g}g^m)\ll O(1/g)$. For error terms of the form
$O(q^{-\min r_j/2})$, note that for any $r$ the number of suitable
$r_1,...,r_m$ s.t. $\min(r_i)=r$ is $O(g^{m-1})$, each contributing
an error term of $g^{-m}q^{-r/2}$, so the total contribution of
these errors is $O(1/g)$. Finally, the number of $r_1,...,r_m$ on
exceptional hyperplanes  is also $O(g^{m-1})$, so the total
contribution of the additional errors is $O(1/g)$.
\end{proof}

Putting together Propositions~\ref{arg},  \ref{pairup}, \ref{main}
we obtain
\begin{thm}\label{finalthm}
Assume that $f_j\in \mathcal S(\R)$ are even  and each $\^f_j(u_j)$
is supported in the range $|u_j|<s_j$, with $\sum s_j<2$. Then
$$\av{W^{(n)}_f}=A(f_1,...,f_n)+O(\log g/g),$$
where
\begin{multline*}
A(f_1,...,f_n)=\sum_{\uf\in\Pi_n} (-2)^{n-|\uf|}\prod_{l=1}^{|\uf|}(|F_l|-1)!\sum_{S\ss\{1,...,|\uf|\}}\lb\prod_{l\in S^c}\uh_{F_l}(0)\rb\cdot\\
\cdot \sum_{S_2\ss S}(-1/2)^{|S_2^c|}\lb\prod_{l\in
S_2^c}U_{F_l}(0)\rb
\lb 2^{|S_2|/2}\sum_{\mathrm{pair}\mbox{ }\mathrm{up}\mbox{ }S_2}\prod_{i=1}^{|S_2|/2}\int_\R \uh_{a_i}(t)\uh_{b_i}(t)|t|\d t-\right. \\
-\frac 12 \sum_{S_3\subsetneq
S_2}2^{|S_3|/2}\sum_{\mathrm{pair}\mbox{ }\mathrm{up}\mbox{ }S_3}
\lb\prod_{i=1}^{|S_3|/2}\int_\R \uh_{a_i}(t)\uh_{b_i}(t)|t|\d t\rb
\cdot\\ \cdot \left. (-2)^{|S_3^c|}\sum_{I\ss
S_3^c}(-1)^{|I|}\int\limits_{\substack{\R_{\ge
0}^{|S_3^c|}\\\sum_{i\in I}t_i\leq \sum_{i\in I^c}t_i-1}}\prod_{i\in
S_3^c}(\uh_i(t)\d t_i)\rb.
\end{multline*}
Here $\uf=\{F_1,...,F_{|\uf|}\}$ ranges over the partitions of
$\{1,...,n\}$, $S$ over the subsets of $\{1,...,|\uf|\}$,
$U_{F_l}(t)=\prod_{k\in F_l}f_k(t)$, $S_2$ ranges over the subsets
of $S$, a pair up sum ranges over partitions
$\{\{a_1,b_1\},...,\{a_{|T|/2},b_{|T|/2}\}\}$ of a set $T$ (it is
empty if $|T|$ is odd), $S_3$ ranges over the proper subsets of
$S_2$ and $I$ ranges over the subsets of $S_3$.
\end{thm}
This coincides with the expression obtained in \cite{Gao} for the
$n$-level density statistics of the family of quadratic
$L$-functions.

\begin{proof}
We go through the verification. From Proposition~\ref{arg},
\begin{multline*}
\ave{W^{(n)}_f} \sim
\sum_{\uf} (-2)^{n-|\uf|}\prod_{l=1}^{|\uf|}(|F_l|-1)!\sum_{S\ss\{1,...,l\}}\lb\prod_{l\in S^c}\uh_{F_l}(0)\rb\cdot\\
\cdot \sum_{S_2\ss S}(-1/2)^{|S_2^c|}\lb\prod_{l\in
S_2^c}U_{F_l}(0)\rb (-1)^{|S_2|} B(U_{l_1},...,U_{l_{|S_2|}})
\end{multline*}
By Proposition~\ref{pairup},
\begin{multline*}
B(U_{l_1},...,U_{l_{|S_2|}})\sim
2^{|S_2|/2}\sum_{\mathrm{pair}\mbox{ }\mathrm{up}\mbox{
}S_2}\prod_{i=1}^{|S_2|/2}
\int_\R \uu_{a_i}(t)\uu_{b_i}(t)|t|\d t+\\
+\sum_{S_3\subsetneq S_2}2^{|S_3|/2}\sum_{\mathrm{pair}\mbox{
}\mathrm{up}\mbox{ }S_3} \prod_{i=1}^{|S_3|/2}\int_\R
\uu_{a_i}(t)\uu_{b_i}(t)|t|\d t \cdot D(S_3^c).
\end{multline*}
Taking into account that the first term above only occurs if $|S_2|$
is even, so that $(-1)^{|S_2|}=1$, gives
\begin{multline*}
\ave{W^{(n)}_f}\sim \sum_{\uf}
(-2)^{n-|\uf|}\prod_{l=1}^{|\uf|}(|F_l|-1)!\sum_{S\ss\{1,...,l\}}
\lb\prod_{l\in S^c}\uh_{F_l}(0)\rb\\
\cdot \sum_{S_2\ss S}(-1/2)^{|S_2^c|}\lb\prod_{l\in
S_2^c}U_{F_l}(0)\rb 
\Big\{ 2^{|S_2|/2}\sum_{\mathrm{pair}\mbox{ }\mathrm{up}\mbox{
}S_2}\prod_{i=1}^{|S_2|/2}
\int_\R \uu_{a_i}(t)\uu_{b_i}(t)|t|\d t \\
+(-1)^{|S_2|}\sum_{S_3\subsetneq
S_2}2^{|S_3|/2}\sum_{\mathrm{pair}\mbox{ }\mathrm{up}\mbox{ }S_3}
\prod_{i=1}^{|S_3|/2}\int_\R \uu_{a_i}(t)\uu_{b_i}(t)|t|\d t \cdot
D(S_3^c) \Big\}  .
\end{multline*}
We also note that the term with $S_3$ only occur if $|S_3|$ is even,
so that we may replace $(-1)^{|S_2|} = (-1)^{|S_3^c|}$. Inserting
Proposition~\ref{main} (which says $\ave{R}\sim D$) gives
\begin{multline*}
\ave{W^{(n)}_f}\sim \sum_{\uf}
(-2)^{n-|\uf|}\prod_{l=1}^{|\uf|}(|F_l|-1)! \sum_{S\ss\{1,...,l\}}
\lb\prod_{l\in S^c}\uh_{F_l}(0)\rb \\
\cdot \sum_{S_2\ss S}(-\frac 12)^{|S_2^c|}\lb\prod_{l\in
S_2^c}U_{F_l}(0)\rb \\
\Big\{ 2^{|S_2|/2}\sum_{\mathrm{pair}\mbox{ }\mathrm{up}\mbox{
}S_2}\prod_{i=1}^{|S_2|/2}
\int_\R \uu_{a_i}(t)\uu_{b_i}(t)|t|\d t \\
-\frac 12  \sum_{S_3\subsetneq
S_2}2^{|S_3|/2}\sum_{\mathrm{pair}\mbox{ }\mathrm{up}\mbox{ }S_3}
\prod_{i=1}^{|S_3|/2}\int_\R \uu_{a_i}(t)\uu_{b_i}(t)|t|\d t
 \\ \cdot (-2)^{|S_3^c| } \sum_{I\ss S_3^c}(-1)^{|I|}
\int\limits_{\substack{t_1,\dots,t_{|S_3^c|}\geq 0\\ \sum t_i\ge 1\\
\sum_{i\in I}t_i\leq \sum_{i\in I^c}t_i-1}} \prod_{i=1}^{|S_3^c|}
 \uu_i(t_i)\d t_i  \Big\}  \;,
\end{multline*}
with a remainder of $O(\log g/g)$. This is exactly the expression
derived in Gao's thesis (see \cite[Theorem II.1]{Gao} or
\cite[Theorem 2.1]{Gaoarxiv}).
\end{proof}

\end{document}